\documentclass[a4paper,reqno,12pt]{amsart}
\usepackage[all, 2cell, knot,]{xy} \UseAllTwocells \SilentMatrices
\usepackage{xspace}
\usepackage{amsmath}
\usepackage{amstext}
\usepackage{amsfonts}
\usepackage[mathscr]{euscript}
\usepackage{amscd}
\usepackage{latexsym}
\usepackage{amssymb}
\usepackage{layout}
\usepackage{amsthm}
\usepackage{amsfonts}
\usepackage{amsxtra}
\usepackage{color}
\usepackage[usenames,dvipsnames,svgnames,table]{xcolor}
\usepackage{graphics}
\usepackage{bm}
\usepackage{tabulary}
\usepackage{fancyhdr}
\usepackage[bookmarks=true,hyperfootnotes=false]{hyperref}
\hypersetup{
			colorlinks=true,
			linkcolor=blue,
			anchorcolor=black,
			citecolor=green,
			urlcolor=black,
}
\theoremstyle{plain}% default, italic fonts for text of the following items
\newtheorem{theorem}{Theorem}[section]
\newtheorem{lemma}[theorem]{Lemma}
\newtheorem{proposition}[theorem]{Proposition}

\theoremstyle{definition}% roman font for text of the following items
\newtheorem{definition}[theorem]{Definition}
\newtheorem{example}[theorem]{Example}

\newtheorem{notation}[theorem]{Notation}
\newtheorem{remark}[theorem]{Remark}

\newcommand{\clC}{\mathcal{C}}
\newcommand{\clD}{\mathcal{D}}
\newcommand{\clE}{\mathcal{E}}

\newcommand{\clI}{\mathcal{I}}

\newcommand{\za}{\alpha}

\newcommand{\zg}{\gamma}

\newcommand{\zD}{\Delta}

\newcommand{\Id}{\operatorname{Id}}

\newcommand{\pr}{\operatorname{pr}}
\newcommand{\nequ}{\mbox{$n$-equivalence}}

\newcommand{\nid}{\noindent}
\newcommand{\bk}{\bigskip}
\newcommand{\mk}{\medskip}

\newcommand{\ovl}[1]{\overline{#1}}
\newcommand{\up}[1]{^{(#1)}}
\newcommand{\lo}[1]{_{(#1)}}
\newcommand{\rw}{\rightarrow}
\newcommand{\Rw}{\Rightarrow}
\newcommand{\lw}{\leftarrow}

\newcommand{\xrw}{\xrightarrow} % use as follows: \xrw{} brackets possibly empty
\newcommand{\xlw}{\xleftarrow} % use as follows: \xlw{} brackets possibly empty
\newcommand{\tiund}[1]{{\times}_{#1}\:}
\newcommand{\pro}[3]{#1\tiund{#2}\overset{#3}{\cdots}\tiund{#2}#1}
\newcommand{\tens}[2]{#1\,\tiund{#2}\,#1}

\newcommand{\oset}[2]{\overset{#1}{#2}}
\newcommand{\ata}{A\tiund{B}A}
\newcommand{\mi}{\text{-}}
\newcommand{\nm}{(n-1)}

\newcommand{\cop}{\textstyle{\,\coprod\,}}

\newcommand{\seqc}[3]{{#1}_{#2},...,{#1}_{#3}}

\newcommand{\dop}[1]{\Delta^{{#1}^{op}}}
\newcommand{\Dop}{\Delta^{op}}
\newcommand{\Dnop}{\Delta^{{n}^{op}}}
\newcommand{\Dmenop}{\Delta^{{n-1}^{op}}}
\newcommand{\cat}[1]{\mbox{$\mathsf{Cat^{#1}}$}}
\newcommand{\Cat}{\mbox{$\mathsf{Cat}\,$}}

\newcommand{\Gpd}{\mbox{$\mathsf{Gpd}$}}
\newcommand{\Af}{A[f]}
\newcommand{\eqr}[1]{\mbox{$\mathsf{EqRel^{#1}}$}}
\newcommand{\cathd}[1]{\mbox{$\mathsf{Cat_{hd}^{#1}}$}}

\newcommand{\catwg}[1]{\mbox{$\mathsf{Cat_{wg}^{#1}}$}}

\newcommand{\muk}{\mu_k}
\newcommand{\hmu}[1]{\hat\mu_{#1}}
\newcommand{\hmuk}{\hat{\mu}_k}

\newcommand{\Nn}{N_{(n)}}
\newcommand{\N}[1]{N_{(#1)}}
\newcommand{\Nb}[1]{N_{(#1)}}
\newcommand{\Nu}[1]{N^{(#1)}}
\newcommand{\di}[1]{d^{(#1)}}

\newcommand{\p}[1]{p^{(#1)}}

\newcommand{\op}[1]{\bar{p}^{(#1)}}

\newcommand{\zgu}[1]{\zg^{(#1)}}
\newcommand{\Set}{\mbox{$\mathsf{Set}$}}

\newcommand{\uk}{\underline{k}}

\newcommand{\us}{\underline{s}}

\newcommand{\funcat}[2]{[\Delta^{{#1}^{op}},#2]}

\newcommand{\nfol}{$n$-fold }
\newcommand{\diag}{\mbox{$diag\,$}}

\begin{document}

\title [\tiny{Homotopically discrete higher categorical structures}]{Homotopically discrete higher categorical structures\\ }

\author[\tiny{Simona Paoli}]{Simona Paoli}
 \address{Department of Mathematics, University of Leicester,
LE17RH, UK}
 \email{sp424@le.ac.uk}

\date{16 May 2016}

\keywords{$n$-fold category, equivalence relation, 0-type}

\subjclass[2010]{18Dxx}

%%%%%%%%%%%%%%%%%%%%%%%%%%%%%%%%%%%%%%%%%%%%%%%%%%%%%%%%%%%%%%%%%%%%%%%%%%%%%

\begin{abstract}
We introduce the notion of homotopically discrete \nfol category as a \nfol generalization of a groupoid with no non-trivial loops. We give two equivalent descriptions of this structure: in terms of a Segal-type model and in terms of iterated internal equivalence relations. We also show that homotopically discrete \nfol categories form a \nfol categorical model of 0-types.
\end{abstract}

%\tableofcontents

\maketitle
\thispagestyle{fancy}

%%%%%%%%%%%%%%%%%%%%%%%%%%%%%%%%%%%%%%%%%%%%%%%%%%%%%%%%%%%%%%%%%%%%%%%%%%%%%%%
\section{Introduction}\label{sec-intro}
A 0-type is a topological space whose homotopy groups are zero in dimension greater than 0. The category of sets is a model of 0-types: to a 0-type we associate the set of its path components. Simplicial sets are well known combinatorial models of spaces \cite{Jard}. In terms of simplicial sets, a 0-type is a simplicial set which is weakly homotopy equivalent to a constant simplicial set: in other words, it is \emph{homotopically discrete}.

In this paper we investigate the notion of homotopically discrete structure in the realm of higher category theory.
In the category $\Cat$ of categories and functors, the notion of homotopically discrete is known: A category equivalent to a discrete one amounts to a groupoid with no non-trivial loops, also called an equivalence relation.

The extension of this notion to higher categories is less straightforward: our interest in developing this extension stems from author's development of \emph{weak globularity} as a new paradigm to weaken higher categorical structures.

In the classical approach to weak $n$-categories, we have sets of cells in dimensions $0$ up to $n$. These cells have compositions which are either associative and unital, as in a strict $n$-category, or they are associative an unital up to coherent isomorphisms. Making this idea precise lead to several different models of higher category. These models use different types of combinatorics, including multi-simplicial sets as in Tamsamani and Simpson \cite{Simp}, \cite{Ta}, (higher) operads as in Batanin  \cite{B}, Leinster \cite{Lein} and Trimble  \cite{Cheng2}, opetopes as in \cite{BD2} \cite{Cheng1},  and several others.

The idea of the weakly globular approach to higher categories if that the discrete structures of the higher cells are replaced by homotopically discrete structures, suitably defined. This represents a new way of weakening a higher categorical structure, leading to new models of weak higher categories suitably equivalent to the classical ones. These new models, which are based on iterated internal categories, have features useful in applications as well as further theoretical developments within higher category theory.

This approach was successfully implemented by the author in the higher groupoidal case. For the modeling of connected $(n+1)$-types the author introduced weakly globular $\mathrm{cat^n\mi groups}$ \cite{Pa}; the latter were shown in \cite{Pa} to be suitable models of $(n+1)$-types which are easier to work with than general $\mathrm{cat^n\mi groups}$  \cite{Lod}; further, they are suitably equivalent to the Tamsamani-Simpson model of higher groupoids \cite{Pa}. The homotopically discrete sub-structures replacing the sets of higher cells in a weakly globular $\mathrm{cat^n\mi group}$ were called in \cite{Pa} 'strongly contractible $\mathrm{cat^n\mi groups}$'.

In  \cite{BP} Blanc and the author developed the notion of weak globularity  to model general $n$-types, via  weakly globular \nfol groupoids. The sets of higher cells were replaced there by 'homotopically discrete \nfol groupoids' (see \cite{BP} for details).

Outside the groupoidal case, for $n=2$, weakly globular double categories were developed in joint work by the author  \cite{PP}, giving a model of weak 2-categories suitably equivalent to bicategories. In a weakly globular double categoy the set of $0$-cells is replaced by a homotopically discrete category (that is, an equivalence relation).

The extension of this program to $n>2$ via the notion of weakly globular $n$-fold category needs a suitable model of homotopically discrete objects, which we call \emph{homotopically discrete $n$-fold categories}. In this paper we define the latter and establish its main properties, while their use in the definition of weakly globular $n$-fold categories will be given in subsequent work \cite{Pa2},\cite{Pa3},\cite{Pa4}.

Homotopically discrete $n$-fold categories are more general than the homotopically discrete $n$-fold groupoids of \cite{BP}. In particular, they are \nfol categories but not in general \nfol groupoids since only some but not all of the $n$ different simplicial directions in the structure are required to be groupoidal. This added generality makes them more suitable to construct  the notion of weakly globular \nfol category.

In this work we present two equivalent descriptions of homotopically discrete $n$-fold categories: one is a multi-simplicial description in the spirit of the Segal-type models of \cite{BP}, built inductively on dimension. The second description is more conceptual and uses an iteration of the notion of internal equivalence relation. We establish in Theorem \ref{the-hom-disc-neq-rel} that these two descriptions are equivalent.

The formal definition of the category $\cathd{n}$ of homotopically discrete \nfol categories is given by induction on dimension (see Definition \ref{def-hom-dis-ncat}) and in particular a homotopically discrete \nfol category is a simplicial object in homotopically discrete $(n-1)$-fold categories. By iterating the nerve construction one obtains the functor
\begin{equation*}
    J_n:\cathd{n}\rw \funcat{n-1}\Cat\;.
\end{equation*}
We show in Lemma \ref{lem-pos-grou-hom-disc} that if $X\in\cathd{n}$, $J_n X$ is levelwise an equivalence relation.

An equivalence relation is categorically equivalent to a discrete category. Similarly one expects a homotopically discrete $n$-fold category to be equivalent to a discrete structure in a higher categorical sense.

This is indeed the case. We define $n$-equivalences between homotopically discrete \nfol categories (Definition \ref{def-hom-dis-ncat-1}): these are a higher dimensional generalization of a functor which is fully faithful and essentially surjective on objects.

There is a simple characterization of $\nequ$s of two homotopically discrete \nfol categories in terms of isomorphisms of the equivalent discrete structures (see Lemma \ref{lem-neq-hom-disc}).

Using this characterization we show that every  homotopically discrete \nfol category $X$ is $n$-equivalent to a discrete $n$-fold category $X^d$ via a 'discretization' map $d:X\rw X^d$. We also show that certain maps called induced Segal maps for a homotopically discrete \nfol category are $\nm$-equivalences (see Proposition \ref{pro-ind-map-hom-disc}). Together with its definition, this makes homotopically discrete \nfol categories a Segal-type model in the sense of \cite{BP}.

In the last part of Section \ref{sec-hom-dis-ncat} we show that homotopically discrete \nfol categories can be described by iterating the notion of internal equivalence relation $\Af$ corresponding to a morphism $f:A\rw B$ in a category $\clC$ with finite limits (see Definition \ref{def-int-eq-rel}). When  $\clC=\Set$ and $f$ is surjective, this is the usual notion of equivalence relation. We define the category $\eqr{n}$ of $\nequ$ relations and we show in Theorem \ref{the-hom-disc-neq-rel} that it is isomorphic to the category $\cathd{n}$.

\bk

\textbf{Acknowledgements}: This work has been supported by a Marie Curie International Reintegration Grant No 256341. I thank the Centre for Australian Category Theory for their hospitality and financial support during August-December 2015, as well as the University of Leicester for its financial support during my study leave. I also thank the University of Chicago for its hospitality and financial support during April 2016.

%%%%%%%%%%%%%%%%%%%%%%%%%%%%%%%%%%%%%%%%%%%%%%%%%%%%%%%%%%%%%%%%%%%%%%%%%%%%

\section{Preliminaries}\label{sec-prelim}
In this section we review some basic simplicial techniques that we will use throughout the paper. The material in this section is well-known, see for instance \cite{Borc}, \cite{Jard}.
\subsection{Simplicial objects}\label{sbs-simp-tech}
Let $\zD$ be the simplicial category and let $\Dnop$ denote the product of $n$ copies of $\Dop$. Given a category $\clC$, $\funcat{n}{\clC}$ is called the category of $n$-simplicial objects in $\clC$ (simplicial objects in $\clC$ when $n=1$).
\begin{notation}\label{not-simp}
    If $X\in \funcat{n}{\clC}$ and $\uk=([k_1],\ldots,[k_n])\in \Dnop$, we shall denote $X ([k_1],\ldots,[k_n])$ by $X(k_1,\ldots,k_n)$, as well as $X_{k_1,\ldots,k_n}$ and $X_{\uk}$.

    %We shall also denote $\uk(1,i)=([k_1],\ldots,[k_{i-1}],1,[k_{i+1}],\ldots,[k_n]) \in \Dnop$ for $1\leq i\leq n$.

    Every $n$-simplicial object in $\clC$ can be regarded as a simplicial object in $\funcat{n-1}{\clC}$ in $n$ possible ways. For each $1\leq i\leq n$ there is an isomorphism
    \begin{equation*}
        \xi_i:\funcat{n}{\clC}\rw\funcat{}{\funcat{n-1}{\clC}}
    \end{equation*}
    given by
    \begin{equation*}
        (\xi_i X)_{r}(k_1,\ldots,k_{n-1})=X(k_1,\ldots,k_{i-1},r,k_{i+1},\ldots,k_{n-1})
    \end{equation*}
    for $X\in\funcat{n}\clC$.
\end{notation}
\begin{definition}\label{def-fun-smacat}
    Let $F:\clC \rw \clD$ be a functor, $\clI$ a small category. Denote
    \begin{equation*}
        \ovl{F}:[\clI,\clC]\rw [\clI,\clD]
    \end{equation*}
    the functor given by
    \begin{equation*}
        (\ovl{F} X)_i = F(X(i))
    \end{equation*}
    for all $i\in\clI$.
\end{definition}
\begin{definition}\label{def-seg-map}
    Let ${X\in\funcat{}{\clC}}$ be a simplicial object in any category $\clC$ with pullbacks. For each ${1\leq j\leq k}$ and $k\geq 2$, let ${\nu_j:X_k\rw X_1}$ be induced by the map  $[1]\rw[k]$ in $\Delta$ sending $0$ to ${j-1}$ and $1$ to $j$. Then the following diagram commutes:
\begin{equation}\label{eq-seg-map}
\xymatrix{
&&&& X\sb{k} \ar[llld]_{\nu\sb{1}} \ar[ld]_{\nu\sb{2}} \ar[rrd]^{\nu\sb{k}} &&&& \\
& X\sb{1} \ar[ld]_{d\sb{1}} \ar[rd]^{d\sb{0}} &&
X\sb{1} \ar[ld]_{d\sb{1}} \ar[rd]^{d\sb{0}} && \dotsc &
X\sb{1} \ar[ld]_{d\sb{1}} \ar[rd]^{d\sb{0}} & \\
X\sb{0} && X\sb{0} && X\sb{0} &\dotsc X\sb{0} && X\sb{0}
}
\end{equation}

If  ${\pro{X_1}{X_0}{k}}$ denotes the limit of the lower part of the
diagram \eqref{eq-seg-map}, the $k$-th Segal map for $X$ is the unique map
$$
\muk:X\sb{k}~\rw~\pro{X\sb{1}}{X\sb{0}}{k}
$$
\noindent such that ${\pr_j\,\muk=\nu\sb{j}}$ where
${\pr\sb{j}}$ is the $j^{th}$ projection.
\end{definition}
\begin{definition}\label{def-ind-seg-map}

    Let ${X\in\funcat{}{\clC}}$ and suppose that there is a map in $\clC$  $\zg: X_0 \rw X^d_0$  such that the limit of the diagram
\begin{equation*}
\xymatrix@C20pt{
& X\sb{1} \ar[ld]_{\zg d\sb{1}} \ar[rd]^{\zg d\sb{0}} &&
X\sb{1} \ar[ld]_{\zg d\sb{1}} \ar[rd]^{\zg d\sb{0}} &\cdots& k &\cdots&
X\sb{1} \ar[ld]_{\zg d\sb{1}} \ar[rd]^{\zg d\sb{0}} & \\
X^d\sb{0} && X^d\sb{0} && X^d\sb{0}\cdots &&\cdots X^d\sb{0} && X^d\sb{0}
    }
\end{equation*}
exists; denote the latter by $\pro{X_1}{X_0^d}{k}$. Then the following diagram commutes, where $\nu_j$ is as in Definition \ref{def-seg-map},
\begin{equation*}
\xymatrix{
&&&& X\sb{k} \ar[llld]_{\nu\sb{1}} \ar[ld]_{\nu\sb{2}} \ar[rrd]^{\nu\sb{k}} &&&& \\
& X\sb{1} \ar[ld]_{\zg d\sb{1}} \ar[rd]^{\zg d\sb{0}} &&
X\sb{1} \ar[ld]_{\zg d\sb{1}} \ar[rd]^{\zg d\sb{0}} && \dotsc &
X\sb{1} \ar[ld]_{\zg d\sb{1}} \ar[rd]^{\zg d\sb{0}} & \\
X^d\sb{0} && X^d\sb{0} && X^d\sb{0} &\dotsc X^d\sb{0} && X^d\sb{0}
}
\end{equation*}
The $k$-th induced Segal map for $X$ is the unique map
\begin{equation*}
\hmuk:X\sb{k}~\rw~\pro{X\sb{1}}{X^d\sb{0}}{k}
\end{equation*}
such that ${\pr_j\,\hmuk=\nu\sb{j}}$ where ${\pr\sb{j}}$ is the $j^{th}$ projection.
\end{definition}
\subsection{$\mathbf{n}$-Fold internal categories}\label{sbs-nint-cat}

Let  $\clC$ be a category with finite limits. An internal category $X$ in $\clC$ is a diagram in $\clC$
\begin{equation}\label{eq-nint-cat}
\xymatrix{
\tens{X_1}{X_0} \ar^(0.65){m}[r] & X_1 \ar^{d_0}[r]<2.5ex> \ar^{d_1}[r] & X_0
\ar^{s}[l]<2ex>
}
\end{equation}
where $m,d_0,d_1,s$ satisfy the usual axioms of a category (see for instance \cite{Borc} for details). An internal functor is a morphism of diagrams like \eqref{eq-nint-cat} commuting in the obvious way. We denote by $\Cat \clC$ the category of internal categories and internal functors.

The category $\cat{n}(\clC)$ of \nfol categories in $\clC$ is defined inductively by iterating $n$ times the internal category construction. That is, $\cat{1}(\clC)=\Cat$ and, for $n>1$, $\cat{n}(\clC)= \Cat(\cat{n-1}(\clC))$.

When $\clC=\Set$, $\cat{n}(\Set)$ is simply denoted by $\cat{n}$ and called the category of \nfol categories (double categories when $n=2$).

%%%%%%%%%%%%%%%%%%%%%%%%%%%%%%%%%%%%%%%%%%%%%%%%%%%%%%%%%%%%%%%%%%%%%%%%%%%%
\subsection{Nerve functors}\label{sus-ner-funct}

There is a nerve functor
\begin{equation*}
    N:\Cat\clC \rw \funcat{}{\clC}
\end{equation*}
such that, for $X\in\Cat\clC$
\begin{equation*}
    (N X)_k=
    \left\{
      \begin{array}{ll}
        X_0, & \hbox{$k=0$;} \\
        X_1, & \hbox{$k=1$;} \\
        \pro{X_1}{X_0}{k}, & \hbox{$k>1$.}
      \end{array}
    \right.
\end{equation*}
When no ambiguity arises, we shall sometimes denote $(NX)_k$ by $X_k$ for all $k\geq 0$.

The following fact is well known:
\begin{proposition}\label{pro-ner-int-cat}
    A simplicial object in $\clC$ is the nerve of an internal category in $\clC$ if and only if all the Segal maps are isomorphisms.
\end{proposition}

By iterating the nerve construction, we obtain the multinerve functor
\begin{equation*}
    \Nn :\cat{n}(\clC)\rw \funcat{n}{\clC}\;.
\end{equation*}
\begin{definition}\label{def-discrete-nfold}
An internal $n$-fold category $X\in \cat{n}(\clC)$ is said to be discrete if $\Nn X$ is a constant functor.
\end{definition}
A basic fact about $\cat{n}(\clC)$ is that each of its objects can be considered as an internal category in $\cat{n-1}(\clC)$ in $n$ possible ways, corresponding to the $n$ simplicial directions of its multinerve. To see this, we use the following lemma, which is a straightforward consequence of the definitions.
\begin{lemma}\label{lem-multin-iff}\
\begin{itemize}
      \item [a)] $X\in\funcat{n}{\clC}$ is the multinerve of an \nfol category in $\clC$ if and only if, for all $1\leq r\leq n$ and $[p_1],\ldots,[p_r]\in\Dop$, $p_r\geq 2$
\begin{equation}\label{eq-multin-iff}
\begin{split}
    &  X(p_1,\ldots,p_r,\mi)\cong\\
    &\cong\pro{X(p_1,\ldots,p_{r-1},1,\mi)}{X(p_1,\ldots,p_{r-1},0,\mi)}{p_r}
\end{split}
\end{equation}
      \item [b)] Let $X\in\cat{n}(\clC)$. For each $1\leq k\leq n$, $[i]\in\Dop$, there is $X_i\up{k}\in\cat{n-1}(\clC)$ with
\begin{equation*}
    \Nb{n-1}X_i\up{k} (p_1,\ldots,p_{n-1})=\Nn X(p_1,\ldots,p_{k-1},i,p_k,\ldots,p_{n-1})
\end{equation*}
    \end{itemize}
\end{lemma}
\begin{proof}\

  a) By induction on $n$. By Proposition \ref{pro-ner-int-cat}, it is true for $n=1$. Suppose it holds for $n-1$ and let $X\in\Cat(\cat{n-1}(\clC))$ with objects of objects (resp. arrows) $X_0$ (resp. $X_1$); denote $(NX)_p=X_p$. By definition of the multinerve
      \begin{equation*}
        (\Nb{n} X)(p_1,\ldots,p_r,\mi)=\Nb{n-1}X_{p_1}(p_2,\ldots,p_r,\mi)\;.
      \end{equation*}
      Hence using the induction hypothesis
\begin{flalign*}
       &\Nb{n}X(p_1...p_r\,\mi)=\Nb{n-1}X_{p_1}(p_2... p_r\,\mi)\cong&
\end{flalign*}
\begin{equation*}
\resizebox{1.0\hsize}{!}{$
\cong \pro{\Nb{n-1} X_{p_1}(p_2... p_{r-1}\,1\,\mi)}
         {\Nb{n-1} X_{p_1}(p_2... p_{r-1}\,0\,\mi)}{p_r}=$}
\end{equation*}
\begin{equation*}
\resizebox{.97\hsize}{!}{
          $=\pro{\Nb{n} X(p_1... p_{r-1}\,1\,\mi)}
         {\Nb{n} X(p_1... p_{r-1}\,0\,\mi)}{p_r}.$} \hspace{10mm}
\end{equation*}
Conversely, suppose $X\in\funcat{n}{\clC}$ satisfies \eqref{eq-multin-iff}. Then for each $[p]\in\Dop$, $X(p,\mi)$ satisfies \eqref{eq-multin-iff}, hence
\begin{equation*}
    X(p,\mi)=\Nb{n-1}X_p
\end{equation*}
for $X_p\in\cat{n-1}(\clC)$. Also, by induction hypothesis
\begin{equation*}
    X(p,\mi)=\pro{X(1,\mi)}{X(0,\mi)}{p}\;.
\end{equation*}
Thus we have the object $X\in\cat{n}(\clC)$ with objects $X_0$, arrows $X_1$ and $X_p=X(p,\mi)$ as above.

\bigskip
b) By part a), there is an isomorphism for $p_r\geq 2$
\begin{flalign*}
       &\Nb{n}X(p_1...p_n)=&
\end{flalign*}
\begin{equation*}
\resizebox{1.0\hsize}{!}{$\pro{\Nb{n}X(p_1...p_{r-1}\, 1 ...p_n)}{\Nb{n}X(p_1...p_{r-1}\, 0 ...p_n)}{p_r}$}\;.
\end{equation*}
In particular, evaluating this at $p_k=i$, this is saying the $(n-1)$-simplicial group taking $(p_1...p_n)$ to $\Nb{n}X(p_1...p_{k-1}\, i ...p_{n-1})$ satisfies condition \eqref{eq-multin-iff} in part a). Hence by part a) there exists $X_i\up{k}$ with
\begin{equation*}
    \Nb{n-1}X_i\up{k}(p_1...p_{n-1})=\Nb{n}X(p_1...p_{k-1}\, i ...p_{n-1})
\end{equation*}
as required.
\end{proof}
\begin{proposition}\label{pro-assoc-iso}
    For each $1\leq k\leq n$ there is an isomorphism $\xi_k:\cat{n}(\clC)\rw \cat{n}(\clC)$ which associates to $X=\cat{n}(\clC)$ an object $\xi_k X$ of $\Cat(\cat{n-1}(\clC))$ with
    \begin{equation*}
        (\xi_k X)_i=X_i\up{k}\qquad i=0,1
    \end{equation*}
    with $X_i\up{k}$ as in Lemma \ref{lem-multin-iff}.
\end{proposition}
\begin{proof}
Consider the object of $\funcat{}{\funcat{n-1}{\clC}}$ taking $i$ to the $\nm$-simplicial object associating to $(\seqc{p}{1}{n-1})$ the object
\begin{equation*}
    \Nb{n}X(p_1...p_{k-1}\, i \, p_{k+1}...p_{n-1})\;.
\end{equation*}
By Lemma \ref{lem-multin-iff} b), the latter is the multinerve of $X_i\up{k}\in \cat{n-1}(\clC)$. Further, by Lemma \ref{lem-multin-iff} a), we have
\begin{equation*}
    \Nb{n-1}X_i\up{k}\cong \pro{\Nb{n-1}X_1\up{k}}{\Nb{n-1}X_0\up{k}}{i}\;.
\end{equation*}
Hence $\Nb{n}X$ as a simplicial object in $\funcat{n-1}{\clC}$ along the $k^{th}$ direction, has
\begin{equation*}
    (\Nb{n}X)_i=
    \left\{
      \begin{array}{ll}
        \Nb{n-1}X_i\up{k}, & \hbox{$i=0,1$;} \\
        \Nb{n-1}(\pro{X_1\up{k}}{X_0\up{k}}{i}), & \hbox{for $i\geq 2$.}
      \end{array}
    \right.
\end{equation*}
This defines $\xi_k X\in\Cat(\cat{n-1}(\clC))$ with
\begin{equation*}
    (\xi_k X)_i=\Nb{n-1}X_i\up{k}\qquad i=0,1\;.
\end{equation*}
We now define the inverse for $\xi_k$. Let $X\in\Cat(\cat{n-1}(\clC))$, and let $X_i=\pro{X_1}{X_0}{i}$ for $i\geq 2$. The $n$-simplicial object $X_{k}$ taking $(p_1,\ldots,p_n)$ to
\begin{equation*}
    \Nb{n}X_{p_k}(p_1...p_{k-1}p_{k+1}...p_n)
\end{equation*}
satisfies condition \eqref{eq-multin-iff}, as easily seen. Hence by Lemma \ref{lem-multin-iff} there is $\xi'_k X\in\cat{n}\clC$ such that $\Nb{n }\xi'_k X=X_k$. It is immediate to check that $\xi_k$ and $\xi'_k$ are inverse bijections.
\end{proof}
\begin{definition}\label{def-ner-func-dirk}
    The nerve functor in the $k^{th}$ direction is defined as the composite
    \begin{equation*}
        \Nu{k}:\cat{n}(\clC)\xrw{\xi_k}\Cat(\cat{n-1}(\clC))\xrw{N}\funcat{}{\cat{n-1}(\clC)}
    \end{equation*}
    so that, in the above notation,
    \begin{equation*}
        (\Nu{k}X)_i=X\up{k}_i\qquad i=0,1\;.
    \end{equation*}
    Note that $\Nb{n}=\Nu{n}...\Nu{2}\Nu{1}$.
\end{definition}
\begin{notation}\label{not-ner-func-dirk}
    When $\clC=\Set$ we shall denote
    \begin{equation*}
        J_n=\Nu{n-1}\ldots \Nu{1}:\cat{n}\rw\funcat{n-1}{\Cat}\;.
    \end{equation*}
\end{notation}
Thus $J_n$ amounts to taking the nerve construction in all but the last simplicial direction. The functor $J_n$ is fully faithful, thus we can identify $\cat{n}$ with the image $J_n(\cat{n})$ of the functor $J_n$.

Given $X\in\cat{n}$, when no ambiguity arises we shall denote, for each $(s_1,\ldots ,s_{n-1})\in\Dmenop$
\begin{equation*}
    X_{s_1,\ldots ,s_{n-1}}=(J_n X)_{s_1,\ldots ,s_{n-1}}\in\Cat
\end{equation*}
and more generally, if $1\leq j \leq n-1$,
\begin{equation*}
    X_{s_1,\ldots ,s_{j}}=(\Nu{j}\ldots \Nu{1} X)_{s_1,\ldots ,s_{j}}\in\cat{n-j}\;.
\end{equation*}
Let $ob : \Cat \clC \rw \clC$ be the object of object functor. The left adjoint to $ob$ is the discrete internal category functor $d$. By Proposition \ref{pro-assoc-iso}  we then have
\begin{equation*}
\xymatrix{
    \cat{n}\clC \oset{\xi_n}{\cong}\Cat(\cat{n-1}\clC) \ar@<1ex>[r]^(0.65){ob} & \cat{n-1}\clC \ar@<1ex>[l]^(0.35){d}\;.
}
\end{equation*}
We denote
\begin{equation*}
\di{n}=\xi^{-1}_{n}\circ d \text{\;\;for\;\;} n>1,\; \di{1}=d \;.
\end{equation*}
Thus $\di{n}$ is the discrete inclusion of $\cat{n-1}\clC$ into $\cat{n}\clC$ in the $n^{th}$ direction.

%%%%%%%%%%%%%%%%%%%%%%%%%%%%%%%%%%%%%%%%%%%%%%%%%%%%%%%%%%%%%%%%%%%%%%%
\subsection{Some functors on $ \Cat$}\label{sbs-funct-cat}

The connected component functor
\begin{equation*}
    q: \Cat\rw \Set
\end{equation*}
associates to a category its set of paths components. This is left adjoint to the discrete category functor
\begin{equation*}
    \di{1}:\Set \rw \Cat
\end{equation*}
associating to a set $X$ the discrete category on that set. We denote by
\begin{equation*}
    \zgu{1}:\Id\Rw \di{1}q
\end{equation*}
the unit of the adjunction $q\dashv \di{1}$.
\begin{lemma}\label{lem-q-pres-fib-pro}
    $q$ preserves fiber products over discrete objects and sends
    equivalences of categories to isomorphisms.
\end{lemma}
\begin{proof}
We claim that $q$ preserves products; that is, given categories
$\clC$ and $\clD$, there is a bijection
\begin{equation*}
    q(\clC\times \clD)=q(\clC)\times q(\clD)\;.
\end{equation*}
In fact, given $(c,d)\in q(\clC\times \clD)$ the map
$q(\clC\times\clD)\rw q(\clC)\times q(\clD)$ given by
$[(c,d)]=([c],[d])$ is well defined and is clearly surjective. On
the other hand, this map is also injective: given $[(c,d)]$ and
$[(c',d')]$ with $[c]=[c']$ and $[d]=[d']$, we have paths in $\clC$
\newcommand{\lin}{-\!\!\!-\!\!\!-}
\begin{equation*}
\xymatrix @R5pt{c \hspace{2mm} \lin \hspace{2mm}\cdots \hspace{2mm}
\lin
\hspace{2mm}c'\\
d \hspace{2mm} \lin \hspace{2mm}\cdots \hspace{2mm} \lin
\hspace{2mm}d' }
\end{equation*}
and hence a path in $\clC\times \clD$
\begin{equation*}
(c,d)\hspace{2mm}\lin\hspace{2mm}\cdots\hspace{2mm}\lin\hspace{2mm}(c',d)
\hspace{2mm}\lin\hspace{2mm}\cdots\hspace{2mm}\lin\hspace{2mm}(c',d')\;.
\end{equation*}
Thus $[(c,d)]=[(c',d')]$ and so the map is also injective, hence it
is a bijection, as claimed.

Given a diagram in $\Cat$ $\xymatrix{\clC\ar_{f}[r] & \clE & \clD
\ar^{g}[l]}$ with $\clE$ discrete, we have
\begin{equation}\label{eq-q-pres-fib-pro}
    \clC\tiund{\clE}\clD=\underset{x\in\clE}{\coprod}\clC_x\times
    \clD_x
\end{equation}
where $\clC_x,\;\clD_x$ are the full subcategories of $\clC$ and
$\clD$ with objects $c,\;d$ such that \;$f(c)=x=g(d)$. Since $q$ preserves
products and (being left adjoint) coproducts, we conclude by
\eqref{eq-q-pres-fib-pro} that
\begin{equation*}
    q(\clC\tiund{\clE}\clD)\cong q(\clC)\tiund{\clE}\,q(\clD)\;.
\end{equation*}
Finally, if $F:\clC\simeq \clD:G$ is an equivalence of categories,
$FG\,\clC\cong\clC$ and $FG\,\clD\cong \clD$ which implies that
$qF\,qG\,\clC\cong q\clC$ and $qF\,qG\,\clD\cong q\clD$, so $q\clC$
and $q\clD$ are isomorphic.
\end{proof}
\bk
The isomorphism classes of objects functor
\begin{equation*}
    p:\Cat\rw\Set
\end{equation*}
associates to a category the set of isomorphism classes of its objects. Notice that if $\clC$ is a groupoid, $p\clC=q\clC$.
\begin{lemma}\label{lem-p-pres-fib-pro}
    $p$ preserves pullbacks over discrete objects and sends
    equivalences of categories to isomorphisms.
\end{lemma}
\begin{proof}
For a category $\clC$, let $m\clC$ be its maximal sub-groupoid. Then $p\clC=qm\clC$. Given a diagram in $\Cat$ $\xymatrix{\clC\ar_{f}[r] & \clE & \clD
\ar^{g}[l]}$ with $\clE$ discrete, we have
\begin{equation*}
    \clC\tiund{\clE}\clD=\underset{x\in\clE}{\coprod}\clC_x\times
    \clD_x\;.
\end{equation*}
Since, as easily seen, $m$ commutes with (co)products, and $m\clE=\clE$, we obtain $m(\clC\tiund{\clE}\clD)=m\clC\tiund{\clE}m\clD$; so by Lemma \ref{lem-q-pres-fib-pro},
\begin{align*}
    p(\clC\tiund{\clE}\clD)&=qm(\clC\tiund{\clE}\clD)=q(m\clC\tiund{\clE}m\clD)
    =qm\clC\tiund{q\clE}qm\clD=p\clC\tiund{\clE}p\clD\;.
\end{align*}
Finally, if $F:\clC\simeq \clD:G$ is an equivalence of categories, $FG\clC\cong \clC$ and $FG\clD\cong\clD$ which implies that $pF\,pG\,\clC\cong p\clC$ and $qF\,qG\,\clD\cong q\clD$, so $q\clC$ and $q\clD$ are isomorphic.
\end{proof}
%%
%%

%%%%%%%%%END PRELIMINARIES%%%%%%%%%%%%%%%%%%%%%%%%%%%%%%%%%%%%%%%%%%%%%%%%%

\section{Homotopically discrete ${n}$-fold categories}\label{sec-hom-dis-ncat}
In this section we give an inductive definition of the category $\cathd{n}$ of homotopically discrete \nfol categories and of $n$-equivalences between them. We then establish the main properties of this structure.

In Lemma \ref{lem-pos-grou-hom-disc} we show that $\cathd{n}$ can be viewed as a diagram of equivalence relations, while in Lemma \ref{lem-neq-hom-disc} we show that $n$-equivalences in $\cathd{n}$ are detected by isomorphisms of their discretizations.

Together with the good behavior of homotopically discrete \nfol categories with respect to pullbacks over discrete objects (Lemma \ref{lem-copr-hom-disc}), this implies that the induced Segal maps in a homotopically discrete \nfol category are $(n-1)$-equivalences (Proposition \ref{pro-ind-map-hom-disc}). This makes $\cathd{n}$ a Segal-type model in a sense similar to \cite{BP}, see also \cite{Ta}.
\subsection{The category of homotopically discrete $\mathbf{n}$-fold categories.}\label{sbs-cat-nfol-cat}
\begin{definition}\label{def-hom-dis-ncat}
    Define inductively the full subcategory $\cathd{n}\subset\cat{n}$ of homotopically discrete \nfol categories.

    For $n=1$, $\cathd{1}=\cathd{}$ is the category of  equivalence relations that is, groupoids equivalent to discrete ones. Denote by $\p{1}=p:\Cat\rw\Set$ the isomorphism classes of object functor.

    Suppose, inductively, that for each $1\leq k\leq n-1$ we defined $\cathd{k}\subset\cat{k}$  such that the following holds:
    \begin{itemize}
      \item [a)] The $k^{th}$ direction in $\cathd{k}$ is groupoidal; that is, if $X\in\cathd{k}$, $\xi_{k}X\in\Gpd(\cat{k-1})$ (where $\xi_{k}X$ is as in Proposition \ref{pro-assoc-iso}).

      \item [b)] There is a functor $\p{k}:\cathd{k}\rw\cathd{k-1}$ making the following diagram commute:
          \begin{equation}\label{eq-p-fun-def}
            \xymatrix{
            \cathd{k} \ar^{\Nu{k-1}...\Nu{1}}[rrr] \ar_{p^{(k)}}[d] &&& \funcat{k-1}{\Cat} \ar^{\bar p}[d]\\
            \cathd{k-1} \ar_{\N{k-1}}[rrr] &&& \funcat{k-1}{\Set}
            }
          \end{equation}
          Note that this implies that $(\p{k}X)_{s_1 ... s_{k-1}}=p X_{s_1 ... s_{k-1}}$ for all $(s_1 ... s_{k-1})\in\dop{k-1}$

    \end{itemize}

    $\cathd{n}$ is the full subcategory of $\funcat{}{\cathd{n-1}}$ whose objects $X$ are such that
    \bigskip
    \begin{itemize}
      \item [(i)]
          \begin{equation*}
            X_s\cong\pro{X_1}{X_0}{s} \quad \mbox{for all} \; s \geq 2.
          \end{equation*}
    In particular this implies that $X\in \Cat(\Gpd(\cat{n-2})) =\Gpd(\cat{n-1})$ and the $n^{th}$ direction in $X$ is groupoidal.
    \medskip
      \item [(ii)] The functor
      \begin{equation*}
        \op{n-1}:\cathd{n}\subset \funcat{}{\cathd{n-1}}\rw\funcat{}{\cathd{n-2}}
      \end{equation*}
      restricts to a functor
      \begin{equation*}
        \p{n}:\cathd{n}\rw\cathd{n-1}
      \end{equation*}
     Note that this implies that $(\p{n}X)_{s_1 ... s_{n-1}}=p X_{s_1 ... s_{n-1}}$ for all $s_1, ..., s_{n-1}\in\dop{n-1}$ and that the following diagram commutes
          \begin{equation}\label{eq-p-fun-def}
            \xymatrix{
            \cathd{n} \ar^{\Nu{n-1}...\Nu{1}}[rrr] \ar_{p^{(n)}}[d] &&& \funcat{n-1}{\Cat} \ar^{\bar p}[d]\\
            \cathd{n-1} \ar_{\N{n-1}}[rrr] &&& \funcat{n-1}{\Set}
            }
          \end{equation}
     \end{itemize}
\end{definition}
\mk
\begin{definition}\label{def-hom-dis-ncat-1}

    Denote by $\zgu{n}_X:X\rw \di{n}\p{n}X$ the morphism given by
    \begin{equation*}
        (\zgu{n}_X)_{s_1...s_{n-1}} :X_{s_1...s_{n-1}} \rw d p X_{s_1...s_{n-1}}
    \end{equation*}
    for all  $(s_1,...,s_{n-1})\in \dop{n-1}$. Denote by
    \begin{equation*}
        X^d =\di{n}\di{n-1}...\di{1}\p{1}\p{2}...\p{n}X
    \end{equation*}
    and by $\zg\lo{n}$ the composite
    \begin{equation*}
        X\xrw{\zgu{n}}\di{n}\p{n}X \xrw{\di{n}\zgu{n-1}} \di{n}\di{n-1}\p{n-1}\p{n}X \rw \cdots \rw X^d\;.
    \end{equation*}
\end{definition}
\begin{notation}\label{not-fiber}
    Given $X\in\cathd{n}$, for each $a,b\in X_0^d$ denote by $X(a,b)$ the fiber at $(a,b)$ of the map
    \begin{equation*}
        X_1 \xrw{(d_0,d_1)} X_0\times X_0 \xrw{\zg\lo{n}\times\zg\lo{n}} X_0^d\times X_0^d\;.
    \end{equation*}
\end{notation}
\begin{definition}\label{def-hom-dis-ncat-1}
Define inductively $n$-equivalences in $\cathd{n}$. For $n=1$, a 1-equivalence is an equivalence of categories. Suppose we defined $\nm$-equivalences in $\cathd{n-1}$. Then a map $f:X\rw Y$ in $\cathd{n}$ is an $n$-equivalence if, for all $a,b \in X_0^d$, $f(a,b):X(a,b) \rw Y(fa,fb)$ and $\p{n}f$ are $\nm$-equivalences.
\end{definition}
\begin{remark}\label{rem-hom-dis-ncat}
By definition, the functor $\p{n}$ sends $n$-equivalences to $\nm$-equivalences. We observe that $\p{n}$ commutes with pullbacks over discrete objects. In fact, if $X\rw Z \lw Y$ is a diagram in $\cathd{n}$ with $Z$ discrete and $X\tiund{Z}Y\in\cathd{n}$, by Definition \ref{def-hom-dis-ncat}
\begin{align*}
    & (\p{n}(X\tiund{Z}Y))_{s_1...s_{n-1}}=p(X_{s_1...s_{n-1}}\tiund{Z}Y_{s_1...s_{n-1}})=\\
   =\  & p X_{s_1...s_{n-1}}\tiund{p Z} p Y_{s_1...s_{n-1}}=(\p{n}X \tiund{\p{n}Z} \p{n}Y)_{s_1...s_{n-1}}
\end{align*}
where we used the fact (Lemma \ref{lem-p-pres-fib-pro}) that $p$ commutes with pullbacks over discrete objects. Since this holds for each $s_1...s_{n-1}$ we conclude that
\begin{equation*}
    \p{n}(X\tiund{Z}Y) \cong \p{n} X \tiund{\p{n} Z} \p{n} Y\;.
\end{equation*}
\end{remark}
\begin{example}\label{ex-hom-disc-neq-rel}
    Let $X\in\cathd{2}$; then $\p{2}X$ is the equivalence relation $p(X_0)[\zg]$ where $\zg:pX_0\rw X^d$, and $X$ has the form
\tiny{
\begin{equation*}
\xymatrix@C10pt{
% First row
% first term empty
& X_{10}\tiund{(\tens{pX_0}{X^d})} X_{10}\tiund{(\tens{pX_0}{X^d})} X_{10} \ar^{}[rr]<1ex> \ar^{}[rr]<-1ex> \ar^{}[d] &&
X_{00}\tiund{pX_0}X_{00}\ar^{}[d]\\
%% Second row
\cdots  \ar^{}[r] \ar_{}[d]<-2ex>\ar^{}[d] &
X_{10}\tiund{(\tens{pX_0}{X^d})}X_{10}\ar^{}[rr]<1ex> \ar^{}[rr]<-1ex>\ar_{}[d]<-2ex>\ar^{}[d] &&
\tens{X_{00}}{pX_{0}}\ar_{}[d]<-2ex>\ar^{}[d]\\
%% Third row
\tens{X_{10}}{X_{00}}\ar_{}[u]<-2ex> & X_{10}\ar_{}[u]<-2ex> && X_{00}\ar_{}[u]<-2ex>
}
\end{equation*}}
\end{example}

\nid The vertical structure is groupoidal and the horizontal  nerve $\Nu{1}X\in\funcat{}{\Cat}$ has in each component an equivalence relation. The horizontal structure is not in general groupoidal; however $\p{2}X$is an equivalence relation. This means that the horizontal arrows in the double category $X$ have inverses after modding out by the double cells. This structure is a special case of what called in \cite{PP} a groupoidal weakly globular double category.

\subsection{Properties of homotopically discrete $\mathbf{n}$-fold categories.}\label{sbs-proper-nfol}
\begin{lemma}\label{lem-pos-grou-hom-disc}
   The functor $J_n:\cat{n}\rw \funcat{n-1}{\Cat}$ restricts to a functor
   \begin{equation*}
    J_n:\cathd{n}\rw \funcat{n-1}{\cathd{}}\;.
   \end{equation*}
\end{lemma}
\begin{proof}
By induction on $n$. For $n=2$ if $X\in\cathd{2}$ then by definition $X_s\in\cathd{}$ for all $s\geq 0$. Suppose the lemma holds for $n-1$ and let $X\in \cathd{n}$. Then for all $s_1\geq 0$, $X_{s_1}\in\cathd{n-1}$ so by induction hypothesis
\begin{equation*}
    (X_{s_1})_{s_2...s_{n-1}}=X_{s_1...s_{n-1}}\in \cathd{}\;.
\end{equation*}
\end{proof}
\begin{lemma}\label{lem-neq-hom-disc}
    A map $f:X\rw Y$ in $\cathd{n}$ is a $\nequ$ if and only if $X^d \cong Y^d$.
\end{lemma}
\begin{proof}
By induction on $n$. For $n=1$, $f$ is a map of equivalence relations, so the statement is is true by Lemma \ref{lem-p-pres-fib-pro}. Suppose the lemma holds for $(n-1)$ and let $f:X\rw Y$ be a $\nequ$ in $\cathd{n}$. Then by definition $\p{n}f$ is a $\nm$-equivalence; therefore by induction hypothesis
\begin{equation*}
    X^d=(\p{n}X)^d=(\p{n}Y)^d=Y^d\;.
\end{equation*}
Conversely, suppose that $f:X\rw Y$ is such that $X^d\cong Y^d$. This is the same as $(\p{n}X)^d=(\p{n}Y)^d$, so, by induction, $\p{n}f$ is a $\nm$-equivalence. This implies that, for each $a,b\in X^d$, $(\p{n}f)(a,b)$ is a $\nm$-equivalence. But
\begin{equation*}
    (\p{n}f)(a,b)=(\p{n-1}f)(a,b)
\end{equation*}
so $(\p{n-1}f)(a,b)$ is a $\nm$-equivalence. This implies that
\begin{equation*}
    X(a,b)^d = (\p{n-1}X(a,b))^d \cong (\p{n-1}Y(fa,fb))^d = Y(fa,fb)^d\;.
\end{equation*}
By induction hypothesis, we deduce that
\begin{equation*}
    f_{(a,b)}:X(a,b) \rw Y(fa,fb)
\end{equation*}
is a $\nm$-equivalence. We conclude that $f$ is a $\nequ$.
\end{proof}
\begin{remark}\label{rem-neq-hom-disc}
    It follows immediately from Lemma \ref{lem-neq-hom-disc} that $\nequ$s in $\cathd{n}$ have the 2-out-of-3 property. In particular this implies that if $X\in\cathd{n}$ the maps $\zg\up{n}:X\rw \di{n}\p{n}X$ and $\zg\lo{n}:X\rw X^d$ are $n$-equivalences.
\end{remark}
\begin{lemma}\label{lem-copr-hom-disc}
    Let $X\xrw{f} Z \xlw{g}Y$ be a diagram in $\cathd{n}$ with $Z$ discrete. Then
    \begin{itemize}
      \item [a)] $X \coprod Y \in \cathd{n}$.\mk

      \item [b)] $X \times Y \in \cathd{n}$.\mk

      \item [c)] $X \tiund{Z} Y \in \cathd{n}$ \quad \text{and} \quad $(X\tiund{Z}Y)^d=X^d\tiund{Z^d}Y^d$.\mk
    \end{itemize}
\end{lemma}
\begin{proof}\

\mk
a) By induction on $n$. It is clear for $n=1$. Suppose it holds for $n-1$ and let $X,Y\in \cathd{n}$. Since $\cathd{n}\subset\funcat{}{\cathd{n-1}}$ and coproducts in functor categories are computed pointwise, for each $s\geq 0$ we have, by induction hypothesis
\begin{equation*}
    (X \cop Y)_s=X_s \cop Y_s\in \cathd{n-1}\;.
\end{equation*}
Since $p$ commutes with coproducts, the same holds for $\p{n}$, thus by induction hypothesis
\begin{equation*}
    \p{n}(X \cop Y)= \p{n} X \cop \p{n} Y\in \cathd{n-1}\;.
\end{equation*}
this proves that $X \cop Y\in \cathd{n}$.

\mk
b) By induction on $n$. It is clear for $n=1$; suppose it holds for $n-1$. Then for each $s\geq 0$, by induction hypothesis
\begin{equation*}
    (X\times Y)_s=X_s \times Y_s\in \cathd{n-1}\;.
\end{equation*}
Since $\p{n}$ commutes with pullbacks over discrete objects (see Remark \ref{rem-hom-dis-ncat}) by the induction hypothesis
\begin{equation*}
    \p{n}(X\times Y)=\p{n}X\times \p{n}Y\in \cathd{n-1}\;.
\end{equation*}
This proves that $X \times Y\in\cathd{n}$.

\mk
c) Since $Z$ is discrete.
\begin{equation*}
    X\tiund{Z}Y=\underset{c\in Z}{\cop}X(c)\times Y(c)
\end{equation*}
where $X(c)$ (resp. $Y(c)$) is the pre-image of $c$ under $f$ (resp. $g$). Since $X(c), \; Y(c) \in \cathd{n}$, from a) and b) it follows that $X\tiund{Z}Y\in\cathd{n}$. Since by Remark \ref{rem-hom-dis-ncat} $\p{n}$ commutes with pullbacks over discrete objects for all $n$, we have
\begin{align*}
    &(X\tiund{Z}Y)^d= p\cdots \p{n}(X\tiund{Z}Y)= \\
    & =p\cdots \p{n}X \tiund{p\cdots \p{n}Z} p\cdots \p{n}Y=X^d\tiund{Z^d}Y^d\;.
\end{align*}
\end{proof}
Given $X\in\cathd{n}$, since $X_0^d$ is discrete and $X_1\in\cathd{n-1}$, by Lemma \ref{lem-copr-hom-disc}, for all $s\geq 2$,
\begin{equation*}
    \pro{X_1}{X_0^d}{s}\in\cathd{n-1}\;.
\end{equation*}
We can therefore consider the induced Segal maps
\begin{equation*}
    \hmu{s}:X_s= \pro{X_1}{X_0}{s}\rw \pro{X_1}{X_0^d}{s}
\end{equation*}
(see Definition \ref{def-ind-seg-map}). Using Lemma \ref{lem-neq-hom-disc} we prove an important property of this map.
\begin{proposition}\label{pro-ind-map-hom-disc}
    Let $X\in\cathd{n}$ and for each $s\geq 2$ consider the map in $\cathd{n-1}$
    \begin{equation*}
        \hat\mu_s:\pro{X_1}{X_0}{s}\rw \pro{X_1}{X^d_0}{s}
    \end{equation*}
    induced by $\zg\lo{n-1}:X_0\rw X_0^d$. Then $\hat\mu_s$ is a $\nm$-equivalence.
\end{proposition}
\begin{proof}
We show this for $s=2$, the case $s>2$ being similar.
By Lemma \ref{lem-neq-hom-disc} it is enough to show that
\begin{equation}\label{eq-ind-map-hom-disc}
    (\tens{X_1}{X_0})^d\cong (\tens{X_1^d}{X^d_0})\;.
\end{equation}
Denote $\p{n-1,j}=\p{j}...\p{n-1}$ for $1\leq j\leq n-1$ and $\p{n-1,n-1}=\p{n-1}$. We claim that
\begin{equation}\label{eq2-ind-map-hom-disc}
    \p{n-1,j}(\tens{X_1}{X_0})=\tens{\p{n-1,j}X_1}{\p{n-1,j}X_0}\;.
\end{equation}
We prove this by induction on $n$. When $n=2$, $X\in\catwg{2}$ so that $p(\tens{X_1}{X_0})\cong p(\tens{X_1}{X_0^d})$. Suppose, inductively, the claim holds for $n-1$. Since $\p{n}X\in\cathd{n-1}$,
\begin{equation*}
    \p{n-1}(\tens{X_1}{X_0})= \tens{\p{n-1}X_1}{\p{n-1}X_0}\;.
\end{equation*}
By induction hypothesis applied to $\p{n}X$ we therefore obtain
\begin{equation*}
\begin{split}
    & \p{n-1,j}(\tens{X_1}{X_0})=\p{n-2,j}\p{n-1}(\tens{X_1}{X_0})= \\
   = \,& \p{n-2,j}(\tens{\p{n-1}X_1}{\p{n-1}X_0})= \\
   =\, & \tens{\p{n-2,j}\p{n-1}X_1}{\p{n-2,j}\p{n-1}X_0} = \\
   =\, & \tens{\p{n-1,j}X_1}{\p{n-1,j}X_0}\;.
\end{split}
\end{equation*}
This proves \eqref{eq2-ind-map-hom-disc}. In the case $j=1$ we obtain
\begin{equation*}
    (\tens{X_1}{X_0})^d=\tens{X_1^d}{X_0^d}\;.
\end{equation*}
Since $\p{n}$ commutes with pullbacks over discrete objects (see Remark \ref{rem-hom-dis-ncat}) we have
\begin{equation*}
\tens{X_1^d}{X_0^d}=(\tens{X_1}{X_0^d})^d
\end{equation*}
so that, from above, we conclude
\begin{equation*}
    (\tens{X_1}{X_0})^d \cong (\tens{X_1}{X_0^d})^d
\end{equation*}
as required.

\end{proof}
\subsection{$\mathbf{n}$-Fold models of 0-types.}\label{sbs-nfol-model}
We end this section with a discussion of the homotopical significance of $\cathd{n}$ as an \nfol model of 0-types.
\begin{definition}\label{def-class-sp-funct}
    The classifying space functor is the composite
    \begin{equation*}
         B:\cathd{n}\xrw{\Nb{n}}\funcat{n}{\Set}\xrw{diag}\funcat{}{\Set}
    \end{equation*}
    where  $\diag$ denotes the multi-diagonal defined by
    \begin{equation*}
    (\diag Y)_k=Y_{k\oset{n}{...}k}
   \end{equation*}
   for $Y\in\funcat{n}{\Set}$ and $k\geq 0$.
\end{definition}
\begin{proposition}\label{pro-sign-cathd}
    If $X\in\cathd{n}$, $B\zg_X:BX\rw BX^d$ is a weak homotopy equivalence. In particular, $BX$ is a 0-type with $ \pi_{i}(BX,x)=0$ for $i>0$ and $\pi_{0}BX= UX^d$ where $UX^d$ is the set underlying the discrete \nfol category $X^d$.
\end{proposition}
\begin{proof}
By induction on $n$. For $n=1$, $X$ is a groupoid with no non-trivial loops, hence $\pi_1(BX,x)=0$ while $\pi_0 BX=UX^d$; suppose the statement holds for $\nm$.

The functor $B$ is also the composite
\begin{equation*}
    B:\cathd{n}\xrw{N_{1}}\funcat{}{\cathd{n-1}}\xrw{\ovl{B}} \funcat{}{\funcat{}{\Set}}\xrw{diag}\funcat{}{\Set}\;.
\end{equation*}
Thus $B\zg_{X}$ is obtained by applying $diag$ to the map of bisimplicial sets $N_1\ovl{B}\zg_{X}$. For each $s\geq 0$ the latter is given by
\begin{equation*}
    (N_1\ovl{B}\zg_{X})_s=B\zg_s: B X_s \rw B X_s^d=B(\p{2}...\p{n}X)_s
\end{equation*}
and this is a weak homotopy equivalence by induction hypothesis.

A map of bisimplicial sets which is a levelwise weak homotopy equivalence induces a weak homotopy equivalence of diagonals (see \cite{Jard}). Hence
\begin{equation*}
    \diag N_1\ovl{B}\zg_X=B\zg_X
\end{equation*}
is a weak homotopy equivalence, as required. Thus $BX$ is weakly homotopy equivalent to $B(\p{2}...\p{n}X)$, which is a 0-type since $\p{2}...\p{n}X\in\cathd{}$. Further $\pi_{0}BX\cong \pi_0 B(\p{2}...\p{n}X) \cong Up\p{2}...\p{n}X \cong UX^d$.

\end{proof}

\section{Higher equivalence relations}\label{highrel}
In this section we give a different description of weakly globular \nfol category via a notion of iterated internal equivalence relation. The notion of internal equivalence relation (Definition \ref{def-int-eq-rel}) associated to a morphism $f:A\rw B$ in a category $\clC$ with finite limits is known. When $\clC=\Set$ and $f:A\rw B$ is surjective, this affords the category $\cathd{}$ and the category $A[f]\in\cathd{}$ corresponding to $f$ has set of connected components given by $q A[f]= p A[f]=B$.

We define $\eqr{n}$ by iterating this notion in $(n-1)$-fold categories in such a way that the target $Y$ of the morphism $f:X\rw Y$ in $\cat{n-1}$ belongs to $\eqr{n-1}$ and $\Nb{n-1}f$ is a levelwise surjection in $\Set$. This surjectivity condition ensures that there is a functor $\p{n}:\eqr{n}\rw\eqr{n-1}$ with $\p{n}X[f]=Y$.

In Theorem \ref{the-hom-disc-neq-rel} we reconcile the definition of $\eqr{n}$ with the definition of $\cathd{n}$ of the previous section.
\begin{definition}\label{def-int-eq-rel}
Let $A\rw B$ be a morphism in a category $\clC$ with finite limits. The diagonal map defines a unique section $s:A\rw\ata$ (so that ${p_{1}s=\Id_{A}=p_{2}s}$ where ${\ata}$ is the pullback of  ${A\xrw{f}B\xlw{f}A}$ and ${p_{1},p_{2}:\ata\rw A}$ are the two projections). The commutative diagram
\begin{equation*}
    \xymatrix{
    \ata \ar[rr]^{p_{1}} \ar[d]_{p_{2}} && A \ar[d]_{f} &&
    \ata \ar[ll]_{p_{2}} \ar[d]^{p_{1}}\\
    A \ar[rr]_{f} && B && A \ar[ll]^{f}
    }
\end{equation*}
defines a unique morphism ${m:(\ata)\tiund{A}(\ata)\rw\ata}$ such that ${p_{2}m=p_{2}\pi_{2}}$ and $p_{1}m=p_{1}\pi_{1}$ where $\pi_{1}$ and ${\pi_{2}}$ are the two projections. We  denote by $\Af$ the following object of $\Cat(\clC)$
\begin{equation*}
\xymatrix{
(\ata)\tiund{A}(\ata)\ar[rr]^<<<<<<<<{m} && \ata \ar@<1.5ex>[rr]^{p_{1}}
\ar@<-.5ex>[rr]^{p_{2}} && A \ar@<1.5ex>[ll]^{s} }
\end{equation*}
\noindent It easy to see that ${\Af}$ is an internal groupoid in $\clC$.

An object of $\cathd{}$ has the form $\Af$ for some surjective map of sets
$f:A\rw B$.
\end{definition}
\begin{definition}\label{def1-int-eq-rel}
    We define $\eqr{n}\subset \cat{n}$ by induction on $n$. For $n=1$, $\eqr{1}=\cathd{}$. Suppose, inductively, we defined $\eqr{n-1}\subset\cat{n-1}$ and let $f:X\rw Y$ be a morphism in $\cat{n-1}$ with $Y\in\eqr{n-1}$ such that, for all $\us\in\Dmenop$, $(\Nb{n-1}f)_{\us}$ is surjective, where $\Nb{n-1}:\cat{n-1}\rw \funcat{n-1}{\Set}$ is the multinerve.

    We define $\eqr{n}$ to be the full subcategory of $\cat{n}$ whose objects have the form $X[f]$ with $f$ as above.
\end{definition}
\begin{remark}\label{rem-int-eq-rel}
    Let $\za:X[f]\rw X'[f']$ be a morphism in $\cathd{n}$, with $f:X\rw Y$ and $f':X'\rw Y'$. For each $\us \in\dop{n-1}$, denote
    \begin{align*}
        &(\Nb{n-1}X)_{\us}=X_{\us}\;,\\
        &(\Nb{n-1}f)_{\us}=f_{\us}\;;
    \end{align*}
    and similarly for $f'_{\us}$. Since $f_{\us}$ is surjective
    \begin{equation*}
        q X_{\us}[f_{\us}]= p X_{\us}[f_{\us}]= Y_{\us}
    \end{equation*}
    and there is a functor
    \begin{equation*}
        X_{\us}[f_{\us}]\rw d Y_{\us}\;.
    \end{equation*}
    We therefore have a commuting diagram in $\Cat$
    \begin{equation*}
        \xymatrix@C=40pt{
        X_{\us}[f_{\us}] \ar[r] \ar_{\za_{\us}}[d] & d Y_{\us} \ar^{\ovl{\za}_{\us}}[d]\\
        X'_{\us}[f'_{\us}] \ar[r] & d Y'_{\us}
        }
    \end{equation*}
    inducing a commuting diagram in $\Set$
    \begin{equation*}
        \xymatrix@C=40pt{
        X_{\us} \ar^{f_{\us}}[r] \ar_{\za_{\us}}[d] &  Y_{\us} \ar^{\ovl{\za}_{\us}}[d]\\
        X'_{\us} \ar^{f'_{\us}}[r] &  Y'_{\us}\;.
        }
    \end{equation*}
    Since this holds for all $\us$, we conclude that there is a commuting diagram in $\cat{n-1}$
    \begin{equation*}
        \xymatrix@C=40pt{
        X \ar^{f}[r] \ar_{\za}[d] &  Y \ar^{\ovl{\za}}[d]\\
        X' \ar^{f'}[r] &  Y'\;.
        }
    \end{equation*}

\end{remark}
\bk
\begin{theorem}\label{the-hom-disc-neq-rel}
   There is an isomorphism of categories $\eqr{n} \cong \cathd{n}$.
\end{theorem}
\begin{proof}
By induction on $n$. For $n=1$, it holds by definition. Suppose this is true for all $k\leq n-1$. Let $X[f]\in\eqr{n}$ with $f: X\rw Y$ a morphism in $\cat{n-1}$ and $Y\in\eqr{n-1}$.

To show that $X[f]\in\cathd{n}$ we need to show that, for all $s_1\geq 0$, $(X[f])_{s_1}\in\cathd{n-1}$ and $\p{n}X[f]\in\cathd{n-1}$ where, for all $\us\in\Dmenop$,
\begin{equation*}
    (\p{n}X[f])_{\us}=p(X[f])_{\us}\;.
\end{equation*}
%
%%, $$ where $f_{s_1}:X_{s_1}\rw Y_{s_1}$ is such that $Y_{s_1}\in\cathd{n-2}=\eqr{n-2}$ (since $Y\in\cathd{n-1}=\eqr{n-1}$) and $(f_{s_1})_{s_2,\ldots,s_{n-1}}=f_{s_1,\ldots,s_{n-1}}$ is surjective.

For all $s_1\geq 0$
\begin{equation*}
    (X[f])_{s_1}=X_{s_1}[f_{1}]\;.
\end{equation*}
There is a morphism $f_{s_1}:X_{s_1}\rw Y_{s_1}$ in $\cat{n-2}$; since $Y\in\eqr{n-1}$, by induction hypothesis $Y\in\cathd{n-1}$, thus by definition $Y_{s_1}\in\cathd{n-2}$ and, by induction hypothesis again, $Y_{s_1}\in\eqr{n-2}$. Further, $(f_{s_1})_{s_2...s_{n-1}}=f_{s_1s_2...s_{n-1}}$ is surjective.
Thus, by definition, $X_{s_1}[f_{s_1}]\in\eqr{n-1}=\cathd{n-1}$, that is,
\begin{equation*}
    (X[f])_{s_1}\in \cathd{n-1}\;.
\end{equation*}
Since $f_{\us}$ is surjective, we have
\begin{equation*}
    p X_{\us}[f_{\us}]=Y_{\us}
\end{equation*}
which implies $\p{n}X[f]=Y\in\cathd{n-1}$, as required.

Conversely, let $X\in\cathd{n}$. Consider the morphism in $\cat{n}$
\begin{equation*}
    \xi_n\zg_X\up{n}:\xi_n X\rw \xi_n \di{n}\p{n}X
\end{equation*}
where $\xi_n$ is as in Proposition \ref{pro-assoc-iso}. At the object of objects level this gives a morphism in $\cat{n-1}$
\begin{equation*}
    f_n X=(\xi_n\zg_X\up{n})_0:(\xi_n X)_0\rw \p{n} X\;.
\end{equation*}
We claim that
\begin{equation}\label{eq1-hom-disc-neq-rel}
    X=(\xi_n X)_0[f_n X]\;.
\end{equation}
We show this by induction on $n$. It is clear for $n=1$ since if $X\in\cathd{}$, $X=X_0[f]$ where $f:X_0\rw pX$.  Suppose it holds for $n-1$. To show \eqref{eq1-hom-disc-neq-rel} it is enough to show that, for each $s\geq 0$,
\begin{equation}\label{eq2-hom-disc-neq-rel}
    X_s=((\xi_n X)_0[f_n X])_s\;.
\end{equation}
But
\begin{equation*}
    ((\xi_n X)_0[f_n X])_s=(\xi_{n-1} X_s)_0[f_{n-1}X_s]
\end{equation*}
where $f_{n-1}X_s=(\xi_{n-1}\zg\up{n-1}_{X_s})_0$

\nid Hence \eqref{eq2-hom-disc-neq-rel} follows by inductive hypothesis applied to $X_s$.

By Lemma \ref{lem-pos-grou-hom-disc}, for each $(s_1,\ldots,s_{n-1})\in\Dmenop$, the map
\begin{equation*}
    (f_n X)_{s_1,\ldots,s_{n-1}}=(\zg_X\up{n})_{s_1,\ldots,s_{n-1},0}: X_{s_1,\ldots,s_{n-1},0}\rw p X_{s_1,\ldots,s_{n-1}}
\end{equation*}
is surjective (as $X_{s_1,\ldots,s_{n-1}}\in\cathd{}$). Also, $\p{n}X\in\cathd{n-1}$ thus by inductive hypothesis $\p{n}X\in\eqr{n-1}$. By \eqref{eq1-hom-disc-neq-rel} and by definition we conclude that $X\in\eqr{n}$.

\end{proof}

\bk
%%%%%%%%%%%%%%%%%%%%%%%%%%%%%%%%%%%%%%%%%%%%%%%%%%%%%%%%%%%%%%%%%%%%%%%%%

\end{document}